\documentclass[a4paper,twoside]{amsart}

\usepackage{amssymb}
\usepackage[mathscr]{eucal}
\usepackage{hyperref}

\newcommand{\thistheoremname}{Theorem}
\newtheorem*{genericthm*}{\thistheoremname}
\newenvironment{namedthm*}[1]{\renewcommand{\thistheoremname}{#1}
   \begin{genericthm*}}{\end{genericthm*}}

\theoremstyle{plain}
\newtheorem{prop}{Proposition}[section]

\newtheorem{cor}[prop]{Corollary}
\newtheorem{lem}[prop]{Lemma}

\theoremstyle{definition}

\newtheorem{lab}[prop]{}

\theoremstyle{remark}
\newtheorem{rem}[prop]{Remark}

\newcommand{\isoto}{\overset{\sim}{\to}}
\renewcommand{\subset}{\subseteq}

\newcommand{\A}{{\mathbb{A}}}
\newcommand{\C}{{\mathbb{C}}}
\renewcommand{\P}{{\mathbb{P}}}
\newcommand{\R}{{\mathbb{R}}}
\newcommand{\Z}{{\mathbb{Z}}}

\newcommand{\scrO}{{\mathscr O}}
\newcommand{\scrP}{{\mathscr P}}

\newcommand{\x}{{\mathtt{x}}}

\DeclareMathOperator{\Div}{Div}
\renewcommand{\div}{\mathop{\rm div}}
\DeclareMathOperator{\Pic}{Pic}
\DeclareMathOperator{\supp}{supp}
\newcommand{\trdeg}{\mathrm{trdeg}}

\renewcommand{\emptyset}{\varnothing}
\newcommand{\ol}{\overline}
\renewcommand{\setminus}{\smallsetminus}
\renewcommand{\epsilon}{\varepsilon}

\newcommand{\sa}{semi-algebraic}

\begin{document}

\title
[Polynomials nonnegative on the cylinder]
{Polynomials nonnegative on the cylinder}

\author
{Claus Scheiderer, Sebastian Wenzel}

\address
 {Fachbereich Mathematik und Statistik \\
 Universit\"at Konstanz \\
 78457 Konstanz \\
 Germany}
\email
 {claus.scheiderer@uni-konstanz.de, wenzelsebastian@gmx.de}



\begin{abstract}
In 2010, Marshall settled the strip conjecture, according to which
every polynomial in $\R[x,y]$, nonnegative on the strip
$[-1,1]\times\R$, is a sum of squares and of squares times $1-x^2$.
We consider affine nonsingular curves $C$ over $\R$ with $C(\R)$
compact, and study the question whether every $f$ in $\R[C][y]$,
nonnegative on $C(\R)\times\R$, is a sum of squares in $\R[C][y]$.
We give an affirmative answer under the condition that $f$ has only
finitely many zeros in $C(\R)\times\R$. For $C$ the circle
$x_1^2+x_2^2=1$, we prove the result unconditionally.
\end{abstract}

%\dedicatory
%  {Dedicated to the memory of Murray Marshall}

\maketitle

%===================================================================%

\section*{Introduction}

A couple of years ago, Murray Marshall \cite{ma} proved that every
polynomial $f\in\R[x,y]$, nonnegative on the strip
$[-1,1]\times\R\subset\R^2$, can be written in the form
$f=s+(1-x^2)t$, where $s,\,t\in\R[x,y]$ are sums of squares of
polynomials. As soon as his result became known, it caused quite a
bit of excitement among the experts. The question had been a
well-known open problem for several years. It originated in a false
claim made in 2001, for which the first author of this present paper
was responsible. At the very end of \cite{ps}, it was announced that
a forthcoming paper would contain a proof of the above statement.
Soon after \cite{ps} had gone into print, the intended proof broke
down, after which the question became known as the \emph{strip
conjecture}. In the years to follow, many people tried in vain to
solve the problem. When Murray surprised us with his success, it was
with great joy and admiration that we studied his elegant arguments.

This paper builds on his ideas. Our initial goal had been to
replace the interval $I=[-1,1]$ by a nonsingular compact real curve
$C(\R)$, and to show that every polynomial $f\in\R[C][y]$,
nonnegative on $C(\R)\times\R$, is a sum of squares in $\R[C][y]$.
However, in this generality we did not succeed.
Following the overall strategy of Murray's argument, there are
several points where new ideas are required. A major problem arises
from the lack of unique factorization in $\R[C]$. This prevents us
from reducing to the case where $f$ has only finitely many zeros in
$C(\R)\times\R$. In general we were unable to overcome this
difficulty, and so we have to assume
that the zero set of $f$ in $C(\R)\times\R$ is finite. When the curve
$C$ is rational, the divisor class group is small enough to get
around this point, using a homological argument. For $C$ the circle
curve, we can therefore prove the full statement without restriction.

The two main results of this paper are thus:

\begin{namedthm*}{Theorem~1}%
Let $C$ be a nonsingular affine curve over $\R$ with $C(\R)$ compact.
Let $V=C\times\A^1$, and let $f\in\R[V]=\R[C][y]$. If $f\ge0$ on
$V(\R)=C(\R)\times\R$, and if $f$ has only finitely many zeros, then
$f$ is a sum of squares in $\R[V]$.
\end{namedthm*}

\begin{namedthm*}{Theorem~2}%
Let $C$ be the plane affine curve over $\R$ with equation
$x_1^2+x_2^2=1$, and let $V=C\times\A^1$. Then every $f\in\R[V]$ with
$f\ge0$ on $V(\R)$ is a sum of squares in $\R[V]$.
\end{namedthm*}

The proof of Theorem~1 (resp.\ Theorem~2) is given in Section~1
(resp.\ Section~2). We also present a generalized version of
Theorem~1 that applies to polynomials nonnegative on $K\times\R$,
where $K$ is a compact semi-algebraic subset of a nonsingular curve.
See Corollary \ref{thm1gend} for the precise statement. We conjecture
that Theorem~1 holds unconditionally for every $f\in\R[V]$
nonnegative on $V(\R)$, even if $f$ has infinitely many real zeros.

The results of this paper are largely contained in the second
author's doctoral thesis \cite{we}. The thesis contains other
generalizations of the strip theorem that we plan to publish
elsewhere.

%===================================================================%

\section{Proof of Theorem~1}

\begin{lab}\label{remtrdeg}%
Marshall's \emph{strip theorem} \cite{ma} provides the first example
of a two-dimensional semi-algebraic set $K\subset\R^n$ for which the
saturated preorder
$$\scrP(K)\>=\>\{f\in\R[\x]\colon f|_K\ge0\}$$
is finitely generated and the ring $B(K)\subset\R[\x]/I_K$ of bounded
polynomial functions on~$K$ has transcendence degree~$\le1$. (Here
$I_K\subset\R[\x]$ is the ideal of polynomials vanishing on~$K$).
Indeed, a~polynomial in $\R[x,y]$ is bounded on $[-1,1]\times\R$ if
and only if it lies in $\R[x]$. To put this remark into perspective,
recall \cite{sch:tams} that $\scrP(K)$ can never be finitely
generated when $\dim(K)\ge3$. Examples of two-dimensional sets $K$
with $\scrP(K)$ finitely generated are known since about 2004, see
\cite{sch:surf}. But all these examples were either compact, or
derived from some compact set in a simple manner. In particular, all
these examples carried plenty of bounded polynomials, in the sense
that the ring $B(K)$ had full trans\-cendence degree two. Before
Marshall's theorem, it was not known whether such examples could
exist with $\trdeg\,B(K)\le1$.%
\end{lab}

\begin{lab}
Our proof is inspired by the strategy of proof in \cite{ma}, although
the details are different in several respects. Let $I=[-1,1]$, and
let $T$ be the preorder in $\R[x,y]$ generated by $1-x^2$. Let
$f\in\R[x,y]$ with $f\ge0$ on $I\times\R$, say $f=a_dy^d+\cdots+a_0$
with $a_i\in\R[x]$ and $a_d\ne0$. To show $f\in T$, Marshall observes
that the leading coefficient $a_d$ is nonnegative on $I$. By
a~standard reparametrization argument he can assume $a_d>0$ on~$I$.
Moreover, by extracting irreducible factors of $f$ with infinitely
many zeros in $I\times\R$, he reduces to the case where $f$ has only
finitely many zeros in $I\times\R$.

Neither step works in our situation. We are considering
$f\in\R[C][y]$, where $C$ is a nonsingular affine curve and $C(\R)$
is compact, and we try to show that $f\ge0$ on $C(\R)\times\R$
implies that $f$ is a sum of squares in $\R[C][y]$. Both reduction
steps would essentially require unique factorization in $\R[C]$. We
get around the first step by using a different approach, based on the
\L ojasiewicz inequality. But we have to make it an assumption that
the zero set of $f$ is finite.

After the initial reduction steps, the key idea in \cite{ma} is to
find a nonzero product $p(x)s(y)$ of two polynomials, with variables
separated, for which $0\le p(x)s(y)\le f(x,y)$ holds on $I\times\R$.
This creates enough room for approximation: One first solves the
problem in polynomials whose coefficients are analytic locally around
$x\in I$, and then uses a refined Weierstra\ss\ approximation
argument to get a global polynomial solution. Our proof essentially
follows this approach, although the details need to be modified, in
particular since we cannot guarantee strict positivity of the leading
coefficient.
\end{lab}

\begin{lab}
Let always $C$ be a nonsingular affine curve over $\R$ whose set
$C(\R)$ of $\R$-points is compact and non-empty. Any $p\in\R[C]$ with
$p\ge0$ on $C(\R)$ is a sum of squares in $\R[C]$, by \cite{sch:mz}
Theorem 4.15). The affine surface $V=C\times\A^1$ has coordinate ring
$\R[V]=\R[C][y]$, the polynomial ring over $\R[C]$ in the
variable~$y$. We will often express elements $0\ne f\in\R[V]$ in the
form $f=\sum_{i=0}^da_iy^i$ with $d\ge0$, $a_i\in\R[C]$ and
$a_d\ne0$. In this case we write $d=\deg_y(f)$ and refer to
$a_d\in\R[C]$ as the \emph{leading coefficient} of~$f$. The zero set
of $f$ in $V(\R)=C(\R)\times\R$ is denoted $Z(f)=\{z\in V(\R)\colon
f(z)=0\}$.
\end{lab}

\begin{lem}\label{leadcoeff}%
Let $0\ne f\in\R[V]$ with $f\ge0$ on $V(\R)$, let $d=\deg_y(f)$, and
let $a\in\R[C]$ be the leading coefficient of~$f$.
\begin{itemize}
\item[(a)]
$d$ is even, and $a\ge0$ on $C(\R)$.
\item[(b)]
If $a>0$ on $C(\R)$, then $Z(f)$ is compact.
\end{itemize}
\end{lem}

\begin{proof}
For any $x\in C(\R)$ with $a(x)\ne0$, restrict $f$ to the line
$\{x\}\times\R\subset V(\R)$ to see that $d$ is even and $a(x)>0$.
Therefore $a\ge0$ on $C(\R)$ by continuity. If $a>0$ on $C(\R)$,
there is a real constant $c>0$ with $a\ge c$ on $C(\R)$. All zeros
$\alpha$ of a polynomial $\sum_{i=0}^da_iy^i$ in $\R[y]$ with
$a_d\ne0$ satisfy $|\alpha|\le\frac1{|a_d|}\sum_{i=0}^d|a_i|$. Since
the coefficients of $f$ are bounded on $C(\R)$, it is therefore clear
that $Z(f)$ is compact.
\end{proof}

Let $f\in\R[V]$ be nonnegative on $V(\R)$ with $Z(f)$ finite. In the
next two lemmas we show that $f$ can be bounded from below by a
product of sums of squares with separated variables. Other than in
\cite{ma} (Lemmas 4.1 and 4.2), we cannot arrange the leading
coefficient of $f$ to be strictly positive. So we have to argue along
a different line.
Recall that a function $C(\R)\to\R$ is called semi-algebraic if its
graph is a semi-algebraic subset of $C(\R)\times\R$.

\begin{lem}\label{sosdrunter1D}%
Let $g\colon C(\R)\to\R$ be a continuous \sa\ function with $g(x)=0$
for only finitely many $x\in C(\R)$. Then there exists $0\ne p\in
\R[C]$ with $p^2\le|g|$ on $C(\R)$.
\end{lem}

\begin{proof}
Let $Z(g)=\{x\in C(\R)\colon g(x)=0\}$, and choose $0\ne q\in\R[C]$
with $q(x)=0$ for every $x\in Z(g)$. By the \sa\ \L ojasiewicz
inequality (\cite{bcr} Corollary 2.6.7) there exist an integer
$N\ge1$ and a real constant $c>0$ with $|q|^N\le c|g|$ on $C(\R)$.
Enlarging $c$ if necessary we can assume that $N=2n$ is even.
So we can take $p=sq^n$, for $s>0$ a small real number.
\end{proof}

\begin{lem}\label{sepsosdrunter}%
Let $f\in\R[V]=\R[C][y]$ with $f\ge0$ on $V(\R)$ and $\deg_y(f)=d$,
and assume $|Z(f)|<\infty$. Given any polynomial $s\in\R[y]$ with
$\deg_y(s)=d$, there exists $0\ne p\in\R[C]$ such that $f(x,y)\ge
p(x)^2s(y)$ for all $(x,y)\in C(\R)\times\R$.
\end{lem}

\begin{proof}
By adding a positive constant to $s$ we may assume that $s>0$ on
$C(\R)$.
Consider $\R$ with its natural embedding in $\P^1(\R)=\R\cup
\{\infty\}$. Since $f$ and $s$ have the same $y$-degree, the map
$C(\R)\times\R\to\R$, $(x,y)\mapsto\frac{f(x,y)}{s(y)}$ extends to a
continuous map $\phi\colon C(\R)\times\P^1(\R)\to\R$, namely by
$\phi(x,\infty)=\frac{a_d(x)}{b_d}$ if $f(x,y)=\sum_{i=0}^d
a_i(x)y^i$, $s(y)=\sum_{i=0}^db_iy^i$.
For $x\in C(\R)$ put
$$g(x)\>:=\>\inf\Bigl\{\frac{f(x,y)}{s(y)}\colon y\in\R\Bigr\}\>=\>
\min\bigl\{\phi(x,y)\colon y\in\P^1(\R)\bigr\}.$$
Then $g\colon C(\R)\to\R$ is a well-defined function with \sa\ graph.
From the second description it is easy to see that $g$ is continuous.
The zeros of $g$ in $C(\R)$ are the zeros of $a_d\in\R[C]$, together
with the projection of $Z(f)\subset C(\R)\times\R$ to $C(\R)$.
Hence $g$ has only finitely many zeros in $C(\R)$, and clearly
$g\ge0$ on $C(\R)$. By Lemma \ref{sosdrunter1D} there exists
$0\ne p\in\R[C]$ with $p^2\le g$ on $C(\R)$. This is the assertion.
\end{proof}

\begin{lab}
In the following let $\scrO_0$ denote the ring of convergent real
power series $\sum_{i\ge0}a_ix^i$ in one variable. This is a
(henselian) discrete valuation ring with residue field~$\R$. As
usual, an element $f$ of a ring $A$ is said to be \emph{psd}
(positive semidefinite) in $A$ if $f$ is nonnegative on the real
spectrum of $A$. By the abstract Nichtnegativstellensatz, it is
equivalent that there is an identity $sf=f^{2n}+t$ with $n\ge0$ and
$s,\,t$ sums of squares in $A$. For the ring $A=\scrO_0[y]$, a
polynomial $f(x,y)=\sum_{i=0}^da_i(x)y^i$ with coefficients
$a_i\in\scrO_0$ is psd in $\scrO_0[y]$ if and only if there exists
$\epsilon>0$ such that $f$ is defined and nonnegative on
$\left]-\epsilon,\epsilon\right[\times\R$.

The following lemma is a particular case of \cite{sch:tams} Lemma
1.8:
\end{lab}

\begin{lem}\label{psdsosscro}%
Every psd element of the polynomial ring $\scrO_0[y]$ is a sum of
squares in $\scrO_0[y]$.
\qed
\end{lem}

\begin{rem}
Lemma \ref{psdsosscro} is a stronger version of \cite{ma} Lemma 4.3,
in that we allow the leading coefficient of the polynomial to lie in
the maximal ideal of $\scrO_0$.
One can in fact show that every psd element in $\scrO_0[y]$ is a sum
of two squares, generalizing also the quantitative part of \cite{ma}
Lemma 4.3. We skip the argument since this fact will not be needed.
\end{rem}

\begin{lab}\label{contanalytfcts}%
On $C(\R)$ there is a natural structure of one-dimensional real
analytic manifold. For any open subset $U\subset C(\R)$, let
$\scrO(U)$ denote the ring of analytic functions $U\to\R$. Given
finitely many points $P_1,\dots,P_r$ in $C(\R)$, let
$A(P_1,\dots,P_r)$ denote the ring of all continuous functions
$C(\R)\to\R$ that are real analytic in suitable neighborhoods of
$P_1,\dots,P_r$.
\end{lab}

\begin{lem}\label{sosinay}%
Given $P_1,\dots,P_r\in C(\R)$, any $f\in\R[V]$ with $f\ge0$ on
$V(\R)$ is a sum of squares in the polynomial ring $A[y]$, where
$A=A(P_1,\dots,P_r)$.
\end{lem}

Here we consider $\R[V]=\R[C][y]$ as a subring of the polynomial
ring $A[y]$ in the natural way.

\begin{proof}
Fix a point $P\in C(\R)$. By Lemma \ref{psdsosscro} there exists an
open neighborhood $U$ of $P$ in $C(\R)$ such that the restriction of
$f$ to $U\times\R$ is a sum of squares in $\scrO(U)[y]$. Hence there
exists a finite covering $C(\R)=U_1\cup\cdots\cup U_m$ by open sets,
together with a sum of squares decomposition of $f|_{U_i\times\R}$ in
the ring $\scrO(U_i)[y]$, for every $i=1,\dots,m$. We can arrange
that $m\ge r$ and $P_i\in U_i$, $P_i\notin\ol{U_j}$ for $i\ne j$ and
$1\le i\le r$, $1\le j\le m$. Let $\beta_i\colon C(\R)\to[0,1]$
($i=1,\dots,m$) be continuous functions forming a partition of unity
and satisfying $\supp(\beta_i)\subset U_i$ for all~$i$.
Now we take the weighted sum of the sum of squares representations of
$f|_{U_i\times\R}$ in $\scrO(U_i)[y]$, using the $\beta_i$ as
weights. The resulting identity is a sum of squares decomposition of
$f$ in $A[y]$.
\end{proof}

Fix a sum of squares decomposition of $f$ in $A[y]$, as in Lemma
\ref{sosinay}. The involved polynomials have coefficients that are
elements of~$A$. We want to approximate these coefficients by regular
functions on $C$. To do this we need some preparations.

\begin{lab}\label{normvgl}%
Let $\R[y]_{\le m}$ be the space of real polynomials of degree at
most~$m$. If $f=\sum_{i=0}^ma_iy^i$ with $a_i\in\R$, we write
$c(f)=\max\{|a_i|\colon i=0,\dots,m\}$ and $||f||=\max\{|f(t)|\colon
-1\le t\le1\}$. Since any two norms on $\R[y]_{\le m}$ are
equivalent, there exist real numbers $\alpha_m,\,\beta_m>0$ such that
$||f||\le\alpha_mc(f)$ and $c(f)\le\beta_m||f||$ for every
$f\in\R[y]_{\le m}$. Clearly one may take $\alpha_m=m+1$.
A concrete value for $\beta_m$ can be deduced easily from Markov's
inequality $||f'||\le m^2||f||$.
For $f,\,g\in\R[y]_{\le m}$, note that $c(fg)\le(m+1)c(f)c(g)$.
\end{lab}

\begin{lem}\label{vglkoeffdiffquad}%
For $i=1,\dots,k$, let $f_i,\,g_i\in\R[y]$ be of degree~$\le m$, and
let $f=\sum_{i=1}^kf_i^2$, $g=\sum_{i=1}^kg_i^2$. If $\epsilon\ge0$
is such that $c(g_i-f_i)\le\epsilon$ for $i=1,\dots,k$, then
$$c(g-f)\>\le\>(m+1)k\epsilon\cdot
\Bigl(\epsilon+2\beta_m\sqrt{||f||}\Bigr).$$
\end{lem}

\begin{proof}
See \ref{normvgl} for notation. Fix $i\in\{1,\dots,k\}$. Using
$g_i^2-f_i^2=(g_i-f_i)(g_i+f_i)$ we get
$$c(g_i^2-f_i^2)\>\le\>(m+1)\,c(g_i-f_i)\,c(g_i+f_i)\>\le\>(m+1)
\epsilon\cdot\bigl(\epsilon+2c(f_i)\bigr)$$
Since $c(f_i)\le\beta_m||f_i||$ and $||f_i||^2\le||f||$, this implies
$$c(g_i^2-f_i^2)\>\le\>(m+1)\epsilon\cdot
(\epsilon+2\beta_m\sqrt{||f||}).$$
Now the assertion follows using $c(g-f)\le\sum_{i=1}^k
c(g_i^2-f_i^2)$.
\end{proof}

\begin{lem}\label{mylem}%
Let $P_1,\dots,P_r\in C(\R)$, and let $\varphi,\,\psi\in
A=A(P_1,\dots,P_r)$ be such that $\psi-\varphi$ is nonnegative on
$C(\R)$ and vanishes at most in $P_1,\dots,P_r$. Then there exists a
regular function $p\in\R[C]$ with $\varphi\le p\le\psi$ on $C(\R)$.
\end{lem}

\begin{proof}
This is similar to Lemma 4.5 in \cite{ma}. If $\varphi<\psi$ on
$C(\R)$, the assertion follows from Weierstra\ss\ approximation.
Otherwise one proceeds by induction on~$r$. Let $P\in C(\R)$
with $\varphi(P)=\psi(P)$, and let $2k>0$ be the vanishing order of
$\psi-\varphi$ at $P$ (note that $\psi-\varphi$ is analytic locally
around~$P$). There exists $t\in\R[C]$ such that $t$ has vanishing
order two at $P$ and $t>0$ on $C(\R)\setminus\{P\}$.
Moreover there exists $q\in\R[C]$ such that $\varphi-q$ and $\psi-q$
vanish at $P$ of order $\ge2k$.
Hence we can define real functions $a,\,b$ on $C(\R)$ by
$$a(x)\>=\>\frac{\varphi(x)-q(x)}{t(x)^k},\quad b(x)\>=\>
\frac{\psi(x)-q(x)}{t(x)^k}\quad(x\in C(\R))$$
Clearly $a,\,b\in A$, we have $a\le b$ on $C(\R)$, and $a(P)<b(P)$.
So by induction there exists $p'\in\R[C]$ with $a\le p'\le b$ on
$C(\R)$. Hence $p:=t^kp'+q$ will do the job.
\end{proof}

\begin{lem}\label{almostsosdecomp}%
Let $f\in\R[V]$ with $f\ge0$ on $V(\R)$ and $\deg_y(f)=d$. Then for
any $0\ne p\in\R[C]$ with $p\ge0$ on $C(\R)$, there is a
decomposition $f=g+\sum_{i=0}^da_iy^i$, where $g\in\R[V]$ is a sum of
squares in $\R[V]$ and $a_0,\dots,a_d\in\R[C]$ satisfy $|a_i|\le p$,
pointwise on $C(\R)$.
\end{lem}

\begin{proof}
Let $A=A(P_1,\dots,P_r)$ be the ring of \ref{contanalytfcts}, where
$P_1,\dots,P_r\in C(\R)$ are the real zeros of $p$. The degree $d$ is
even, say $d=2m$. By Lemma \ref{sosinay} there is a sum of squares
decomposition $f=f_1^2+\cdots+f_k^2$ with $f_i\in A[y]$, say
$$f_i\>=\>\sum_{j=0}^mb_{ij}y^j$$
with $b_{ij}\in A$.
Let $\lambda>0$ be a real parameter that will be adjusted later. By
Lemma \ref{mylem} there exist regular functions $q_{ij}\in\R[C]$
such that $\bigl|q_{ij}-b_{ij}\bigr|\le\lambda p$ on $C(\R)$, for
$1\le i\le k$ and $0\le j\le m$.
Put
$$g_i\>:=\>\sum_{j=0}^mq_{ij}y^j$$
($i=1,\dots,k$), let $g:=g_1^2+\cdots+g_k^2\in\R[V]$ and write $f-g=
\sum_{i=0}^da_iy^i$ with $a_i\in\R[C]$. We can estimate the $|a_i|$
as follows. Let $\gamma>0$ be a real number such that $f(x,y)\le
\gamma^2$ for $x\in C(\R)$ and $|y|\le1$. Using Lemma
\ref{vglkoeffdiffquad} we get
$$|a_i(x)|\>\le\>(m+1)k\cdot\lambda p(x)\cdot
\Bigl(\lambda p(x)+2\beta_m\gamma\Bigr),\quad x\in C(\R).$$
For $\lambda>0$ sufficiently small, the right hand side is less or
equal to $p(x)$, uniformly for all $x\in C(\R)$. This proves the
lemma.
\end{proof}

\begin{lab}\label{endproofthm1}%
We now give the proof of Theorem~1. So let $C$ be a nonsingular
affine curve over $\R$ with $C(\R)$ compact, and let $V=C\times\A^1$.
Let $f\in\R[V]$ with $f\ge0$ on $V(\R)$ and with only finitely many
zeros in $V(\R)$, and let $\deg_y(f)=2m$. Fix a strictly positive
polynomial $s\in\R[y]$ with $\deg(s)=2m$, for example $s=y^{2m}+1$.
By Lemma \ref{sepsosdrunter} there exists a sum of squares $p\ne0$ in
$\R[C]$ with $ps\le f$ on $V(\R)$. Let
$$t\>:=\>\sum_{i=0}^{2m}y^i+2\sum_{j=0}^my^{2j}\>=\>3+y+3y^2+y^3+
\cdots+3y^{2m},$$
a polynomial in $\R[y]$ with $\deg(t)=2m$.
There is a real number $c>0$ such that the polynomial $s-ct$ is
nonnegative (and hence a sum of squares) in $\R[y]$, since $s$ is
strictly positive and $\deg(t)=\deg(s)$. By Lemma
\ref{almostsosdecomp}, applied to $f-ps\in\R[V]$ and
$\frac c3p\in\R[C]$, there is a sum of squares $g$ in $\R[V]$ such
that
$$f-ps\>=\>g+\sum_{i=0}^{2m}b_iy^i$$
with $b_i\in\R[C]$ for which $3|b_i|\le cp$ holds on $C(\R)$
($i=0,\dots,2m$). We now mimick Marshall's marvelous decomposition
(last two pages of \cite{ma}), thereby proving that $f$ is a sum of
squares: We have $f=g+h_1+h_2$ where $h_1=(s-ct)p$ and
$h_2=ctp+\sum_{i=0}^{2m}b_iy^i$. Clearly $h_1$ is a sum of squares in
$\R[y]$. And $h_2$ is a sum of squares in $\R[V]$, since $h_2$ is the
sum of the following polynomials:
$$b_0-b_1+2cp,\quad\Bigl(b_{2m}-b_{2m-1}+2cp\Bigr)y^{2m},$$
$$\Bigl(b_i+cp\Bigr)\,y^{i-1}(1+y+y^2)\quad(\text{for }0<i<2m,\
i\text{ odd}),$$
$$\Bigl(b_i-b_{i-1}-b_{i+1}+cp\Bigr)y^i\quad(\text{for }0<i<2m,\
i\text{ even}).$$
Each of these is a psd polynomial in $y$, times an element of $\R[C]$
that is nonnegative on $C(\R)$ (and that is hence, by \cite{sch:mz},
a sum of squares in $\R[C]$). The reason is $3|b_i|\le cp$ on $C(\R)$
for all~$i$. Theorem~1 is proved.
\qed
\end{lab}

\begin{rem}\label{norealsings}%
In Theorem~1 we may relax the hypothesis by allowing the curve $C$ to
have singularities in nonreal points. The proof given above carries
over verbatim to this more general case.
\end{rem}

\begin{lab}\label{prepsthm1gen}%
Generalizing the setup of Theorem~1, one may ask if the compact curve
$C(\R)$ can be replaced by a compact semi-algebraic set $K$ on some
curve. Hereby sums of squares need to be replaced by elements of a
suitable preorder. Such generalizations are indeed possible, as we'll
indicate now. We are content with a straightforward formulation
and do not strive for the most general version.

Let $C$ be a nonsingular affine curve over $\R$, and let $K\subset
C(\R)$ be a compact semi-algebraic subset without isolated points.
By \cite{sch:mz} Theorem 5.22, the saturated preorder
$$\scrP(K)\>:=\>\bigl\{p\in\R[C]\colon p\ge0\text{ on }K\bigr\}$$
in $\R[C]$ can be generated by a single element $h\in\R[C]$.
Indeed,
there exists $h\in\R[C]$ with $K=\{x\in C(\R)\colon h(x)\ge0\}$ such
that $h$ has vanishing order~$1$ at every boundary point of $K$, and
has no other zeros in $C(\R)$. Any such $h$ will generate the
preorder $\scrP(K)$ in $\R[C]$, according to \cite{sch:mz}.
\end{lab}

\begin{cor}\label{thm1gend}%
Let $C$ be a nonsingular affine curve over $\R$, and let $K\subset
C(\R)$ be a compact semi-algebraic set without isolated points.
Let $h\in\R[C]$ generate the preorder $\scrP(K)$ in $\R[C]$.
If $f\in\R[C\times\A^1]=\R[C][y]$ satisfies
$f\ge0$ on $K\times\R$, and if $f$ has only finitely many zeros in
$K\times\R$, there are sums of squares $g_0,\,g_1$ in $\R[C][y]$ such
that $f=g_0+g_1h$.
\end{cor}

\begin{proof}
Corollary \ref{thm1gend} could be proved by inspecting each step in
the proof of Theorem~1 and replacing it a suitably generalized
version. It is however easier to obtain \ref{thm1gend} as a direct
corollary to Theorem~1:

Let $C'$ be the affine curve with coordinate ring $\R[C']=\R[C][z]/
(z^2-h)$, and let $C'\to C$ be the natural morphism. Then $f$,
considered as an element of $C'[y]$, is nonnegative on $C'(\R)
\times\R$ and has only finitely many zeros there. Since the curve
$C'$ has no real singularities, it follows from Theorem~1 (observe
Remark \ref{norealsings}) that $f$ is a sum of squares in $R[C']$. So
we have $f=\sum_i(a_i+b_iz)^2$ with $a_i,\,b_i\in\R[C][y]$ (and
$z^2=h$). Expanding this expression shows $f=\sum_ia_i^2+h\sum_i
b_i^2$ in $\R[C]$.
\end{proof}

%===================================================================%

\section{Proof of Theorem~2}

In the following let $C$ be the plane real curve with equation
$x_1^2+x_2^2=1$. Let $V=C\times\A^1$, so $\R[V]=\R[C][y]$ is the
polynomial ring over $\R[C]=\R[x_1,x_2]/(x_1^2+x_2^2-1)$.

\begin{lem}\label{factorizeoncurve}%
Let $0\ne p\in\R[C]$ with $p\ge0$ on $C(\R)$. There exist
$p_1,\,p_2\in\R[C]$, both nonnegative on $C(\R)$, with $p=p_1p_2$ and
such that $p_1$ has only real zeros on $C$, while $p_2$ has no real
zeros.
\end{lem}

\begin{proof}
Let $\xi\in C(\R)$ with $p(\xi)=0$. The vanishing order of $p$ at
$\xi$ is even, so by induction it suffices to show that there exists
$q\in\R[C]$ with a double zero in $\xi$ and with no other zeros in
$C$. But this is clear, one can take $q$ to be the tangent to $C$
at~$\xi$.
\end{proof}

Note that there is no analogue of Lemma \ref{factorizeoncurve} when
the curve $C$ has positive genus.

\begin{lem}\label{standardlem}%
Let $f\in\R[V]$ and $b\in\R[C]$ be sums of squares, and assume that
$b$ has only real zeros on $C$. If there is $g\in\R[V]$ with $f=bg$,
then $g$ is a sum of squares in $\R[V]$ as well.
\end{lem}

\begin{proof}
We have $b^2g=bf=\sum_ih_i^2$ with $h_i\in\R[V]$. Since $b$ has only
real zeros, we have $h_i=bg_i$ for suitable $g_i\in\R[V]$, see
\cite{sch:tams} Lemma~0.2, and so $g=\sum_ig_i^2$.
\end{proof}

\begin{lab}
We need a small argument involving divisor class groups. Let $X$ be
an irreducible variety over a field $k$. By $Cl(X)$ we denote the
codimension one Chow group of $X$, i.e.\ the group of Weil divisors
on $X$ modulo rational equivalence. As usual let $\Pic(X)$ be the
Picard group of Cartier divisors on $X$ modulo linear equivalence.
There is a natural map $\Pic(X)\to Cl(X)$ which in general is neither
injective nor surjective. When $X$ is nonsingular (or more generally
locally factorial), the map $\Pic(X)\to Cl(X)$ is a group
isomorphism. See e.g.\ \cite{fu} section~2.1 for these notions and
facts.
\end{lab}

\begin{lab}\label{realnonreal}%
We only need these concepts for nonsingular irreducible varieties $X$
over $k=\R$. Given such $X$ let $\Div(X)$ be the group of Weil
divisors on $X$, i.e.\ the free abelian group on the irreducible
codimension one subvarieties $Y$ (called prime divisors) of~$X$.
A~prime divisor $Y$ is said to be \emph{real} if $Y(\R)$ is Zariski
dense in $Y$, otherwise \emph{nonreal}. Given a Weil divisor
$D=\sum_{i=1}^rY_i$ on $X$ with prime divisors $Y_1,\dots,Y_r$, we
let $D(\R)=\bigcup_{i=1}^rY_i(\R)$.
\end{lab}

\begin{lem}\label{cptdivzero}%
Let $C$ be the plane affine curve $x_1^2+x_2^2=1$ over $\R$. Any Weil
divisor $D\ge0$ on $V=C\times\A^1$ for which $D(\R)$ is compact is
rationally equivalent to zero.
\end{lem}

\begin{proof}
Pullback of divisors via the projection map $C\times\A^1\to C$
induces a group isomorphism $Cl(C)\isoto Cl(C\times\A^1)$, see
\cite{fu} Theorem 3.3.
The inverse map $Cl(C\times\A^1)\to Cl(C)$ is given by intersecting a
divisor on $C\times\A^1$ with the $1$-cycle $C\times\{\xi\}$, for
$\xi\in\A^1(\R)=\R$ an arbitrary point (see \cite{fu} 3.3.1). Hence
the class of the divisor $D$ in the assertion is a $2$-fold in
$Cl(V)$. Since $Cl(V)\cong Cl(C)=\Z/2$, this proves the claim.
\end{proof}

\begin{lab}\label{endproofthm2}%
We give the proof of Theorem~2. Let $f=f(x,y)\in\R[V]$ with $f\ge0$
on $V(R)$. We have to show that $f$ is a sum of squares in $\R[V]$.
Write $f=\sum_{i=0}^da_iy^i$ with $a_i\in\R[C]$ and $a_d\ne0$. By
Lemma \ref{leadcoeff}, $d$ is even and $a_d\ge0$ on $C(\R)$. By Lemma
\ref{factorizeoncurve} we can write $a_d=bc$ with $b,\,c\in\R[C]$,
such that $b$ has only real zeros on~$C$, and such that $b\ge0$ and
$c>0$ on $C(\R)$. Multiplying $f$ with $b^{d-1}$ gives
$$b(x)^{d-1}f(x,y)\>=\>g(x,\,b(x)y),\quad(x,y)\in\C(\R)\times\R,$$
where $g\in \R[V]$ is defined by $g=cy^d+\sum_{i=0}^{d-1}a_ib^{d-1-i}
y^i$.
Clearly, $g\ge0$ on $V(\R)$ as well,
and the leading coefficient $c$ of $g$ is strictly positive on
$C(\R)$. It suffices to prove that $g$ is a sum of squares in
$\R[V]$. Indeed, this implies that $b^{d-1}f$ is a sum of squares in
$\R[V]$,
and by Lemma \ref{standardlem} we conclude that $f$ itself is a sum
of squares in $\R[V]$.

So we can assume that the leading coefficient of $f$ is strictly
positive on $C(\R)$. By Lemma \ref{leadcoeff}(b), the real zero set
$Z(f)\subset V(\R)$ of $f$ is compact. For every real prime divisor
$Y$ on $V$, the vanishing order of $f$ along $Y$ is even. Therefore
we can decompose the Weil divisor $\div(f)$ on $V$ as $\div(f)=2D+E$,
in such a way that every irreducible component of $D$ is real and
every irreducible component of $E$ is nonreal (see
\ref{realnonreal}).

Since $Z(f)$ is compact, it follows that $D(\R)$ is compact as well.
By Lemma \ref{cptdivzero} (and since $\Pic(V)\cong Cl(V)$), this
implies that $D=\div(g)$ for some rational function $g\ne0$ on $V$.
Since $V$ is nonsingular, hence normal, we have $g\in\R[V]$. This
means we have a product decomposition $f=g^2h$ with $g,\,h\in\R[V]$,
and every irreducible component of $\div(h)=E$ is nonreal. Therefore
$h$ has only finitely many zeros in $V(\R)$. By Theorem~1, $h$ is a
sum of squares in $\R[V]$. Hence so is $f$, and Theorem~2 is proved.
\qed
\end{lab}

\begin{rem}
Theorem~2 provides the first example of an affine algebraic surface
$V$ over $\R$ for which $\rm psd=sos$ holds in $\R[V]$, and for which
the ring $B(V)\subset\R[V]$ of bounded polynomials has
$\trdeg\,B(V)\le1$ (c.f.\ Remark \ref{remtrdeg}).
\end{rem}

We may generalize Theorem~2 slightly:

\begin{cor}\label{thm2gen}%
Let $X$ be any nonsingular affine rational curve over $\R$ for which
$X(\R)$ is compact. Then $\rm psd=sos$ holds on the surface
$X\times\A^1$.
\end{cor}

\begin{proof}
For $X(\R)=\emptyset$ the assertion is clear. The only examples of
such $X$ with $X(\R)\ne\emptyset$ are of the form $X=C\setminus Z$
where $Z$ is a finite set of nonreal closed points of $C$ (and $Z$ is
conjugation-invariant, depending on the view point). Choose
$h\in\R[C]$ such that $Z$ is the set of zeros of $h$ in $C$. Then
$\R[X]=\R[C]_h$, the ring of fractions. If $f\in\R[X\times\A^1]=
(\R[C]_h)[y]$ is nonnegative on $X(\R)\times\R$, write
$f=\frac g{h^r}$ with $r\ge0$ even and $g\in\R[C][y]$. Then $g\ge0$
on $C(\R)\times\R$, so $g$ is a sum of squares in $\R[C][y]$ by
Theorem~2. Hence $f$ is a sum of squares in $\R[X\times\A^1]$.
\end{proof}

%===================================================================%

\end{document}